\documentclass[aps, prd, onecolumn, amssymb, amsmath, amsfonts]{revtex4-2}

\usepackage{amsthm}
\usepackage{graphicx}
\usepackage{hyperref}[hidelinks]

\makeatletter
\gdef\th@remark{%
	\thm@headfont{\small\scshape}%
}
\makeatother
\theoremstyle{remark}
\newtheorem{rmk}{Remark}

\newtheorem{pro}{Proposition}
\begin{document}

\title{Dynamics of a multilink wheeled vehicle: partial solutions and unbounded
	speedup}

\author{E.M. Artemova}
\email{liz-artemova2014@yandex.ru}
\author{I.A. Bizyaev}
\email{bizyaevtheory@gmail.com}

\affiliation{Ural Mathematical Center, Udmurt State University, ul. Universitetskaya 1, 426034 Izhevsk, Russia}

\begin{abstract}
%% Text of abstract
 A mathematical model featuring the motion of a multilink
wheeled vehicle is developed using a nonholonomic model.
A detailed analysis of the inertial motion is made.  Fixed points
of the reduced system are identified, their stability is analyzed,
and invariant manifolds are found. For the case of three platforms (links),
a phase portrait for motion on an invariant manifold is shown and trajectories of the
attachment points of the wheel pairs of the three-link vehicle are presented.
In addition, an analysis is made of motion in the case where the leading platform has a
rotor whose angular velocity is a periodic function of time.
The existence of trajectories for which one of the velocity components
increases without bound is established, and the asymptotics for it is found.
\end{abstract}

	\maketitle

\tableofcontents

\section{Introduction}
\label{}

This paper continues a series of studies of the dynamics of various wheeled vehicles using a nonholonomic model. We recall that this model assumes the absence of slipping at the points of contact of the wheels with the supporting plane.  It was shown in \citep{borisov2015hadamard} that, if the wheel has the form of a circle with one point of contact with the supporting plane, it is equivalent to a skate and the absence of slipping in the direction perpendicular to the plane of the skate can be described by a nonholonomic constraint. In what follows, we will assume this condition to be satisfied. Of course, this approach does not take into account more complicated wheel designs such as, for example, the omniwheel \citep{Borisov2015Dynamics} or the origami wheel \citep{Fonseca2020Nonlinear}.

The simplest and best studied wheeled vehicle is a rigid body with a wheel pair fastened
to it. This system is equivalent to the Chaplygin sleigh \citep{chaplygin1912}. A detailed
analysis of the motion of the Chaplygin sleigh was carried out by Carath\'{e}odory
\citep{caratheodory1933}, and various generalizations of this problem were considered in
\citep{borisov2009mamaev, Czudkova2013Nonholonomic}. The problems of controlling such a wheeled vehicle are interesting from the viewpoint of practical application. The motion of the Chaplygin sleigh under the
action of periodic changes in the mass distribution was investigated in \citep{osborne2005}.
In \citep{Bizyaev2021Normal} (see also references therein), a description is given of an
interesting dynamical phenomenon,
a nonlinear nonholonomic acceleration under the action of bounded periodic excitation
generated by the prescribed motion of structural elements (point masses, rotors etc.).
Various versions of the problem of the optimal control of such a vehicle by controlling the
translational and angular velocities were addressed in \citep{sachkov2010maxwell, Ardentov2023OF}.
The influence of dry friction on the motion of the symmetric and asymmetric Chaplygin sleigh was
investigated in \citep{Karapetyan2019, Shamin2022}, and the influence of viscous friction was examined in \citep{fedonyuk2020locomotion}.

More elaborate designs are those of two-link and three-link
vehicles consisting of two or three platforms with rigidly attached wheel pairs.
The most general asymmetric two-link wheeled vehicle of such type is usually called a {\it
	Roller Racer}. The inertial and controlled motion of a Roller Racer was
examined in detail in \citep{Bizyaev2017roller, BizyaevMamaev2023, Krishnaprasad2001,Chakon2017Analysis, Halvani2022Nonholonomic, bazzi2017motion}. The inertial motion of a symmetric three-link vehicle, namely,
the special case in which the system is integrable was studied in \citep{Artemova2020}.
The motion of a Roller Racer where the point of attachment of the wheel pair coincides with
the point of coupling of the platforms, with periodically moving point masses on the leading
platform, was examined in \citep{Mikishanina2021,Mikishanina2023problem}. A stability
analysis of the straight-line motion of a Roller Racer on a vibrating plane was carried out
in \citep{Artemova2022}. 

The snakeboard is an intermediate model between the Roller Racer and the three-link robot. It is a three-link robot on whose central platform there is no wheel pair. The free motion (without controls) of a snakeboard was studied in \citep{borisov2019invariant}, and the controllability of this system by the Rashevsky-Chow theorem was shown in \citep{ostrowski1994nonholonomic}. The motion of a snakeboard was also studied, for example, in \citep{bloch1996nonholonomic, shammas2007towards, janova2010streetboard}. An analysis of the controlled motion by controlling one or several angles between the links of a three-link robot was carried out in \citep{dear2016locomotive, Jing, Kilin2022, shammas2007geometric}. It is well known that ``singular’’ configurations \citep{Dear, nansai2018generalized} arise when the angles between the platforms in multilink systems are chosen as controls. When the robot passes through these configurations, the constraint reactions begin to grow without bound, so that the nonholonomic model becomes inapplicable. Therefore \citep{yona2019wheeled}, a hybrid model was proposed  with switching between the nonholonomic model and the model taking into account sliding obeying the law of Coulomb friction, for a symmetric three-link robot.

The dynamics of a wheeled vehicle consisting of three or more platforms is explored in
\citep{Naranjo2015}, where equations of motion are obtained for an articulated $N$-trailer
robot performing inertial motion on a plane and it is assumed that the trailer consists of
$N$ identical platforms. The control of such wheeled systems, which consist of $N$ coupled
platforms, with a wheel pair fastened to each of them, was discussed, for example, in
\citep{transeth2008modelling,Tanaka}. The controlled motion of such robots with passive (free) platforms and an experiment in controlling  an $N$-trailer robot (with $N=4$) were performed in \citep{dear2020locomotion}.

Nonetheless, the dynamics of an $N$-link ($N>3$) wheeled vehicle remains
largely underexplored. This is partially due to the problem of deriving equations of motion
and reducing them to a form convenient for investigation. For example, incorporating a platform (link) leads to an additional nonholonomic constraint, which, in turn, leads to addition of an undetermined Lagrange
multiplier to the equations of motion (see, e.g., \citep{Kozlov2008Gauss, Ghori1999Principles}).
In this paper, we choose quasi-velocities in such a way that
nonholonomic constraints are given by equality of these quasi-velocities to zero.
Such a parameterization is described in detail in \citep{borisov2015symmetries} and is called
a {\it natural parameterization}, but it was used already by Hamel \citep{Hamel}.
A natural parameterization allows undetermined multipliers to be eliminated from equations
of motion. As quasi-velocities we choose the translational and angular velocities of the platforms in the moving coordinate system. Moreover, for the system under consideration it turned out that part of constraints
can be taken into account in the Lagrangian function prior to substitution into the equations of
motion.
All this has made it possible to obtain equations of motion  for the system
in a sufficiently compact form.

In the general case, the straight-line motion (without control) for wheeled systems
can lose stability \citep{BizyaevMamaev2023, Troger1984Nonlinear} as the parameters are varied.
In this paper, it is shown that, for the system considered,
the straight-line motion is stable for an arbitrary number of links and for any parameter values.
It is also shown that adding a periodic control action (a rotor) leads to an unbounded
speedup of the wheeled vehicle in a neighborhood of the solution corresponding
to straight-line motion, and the asymptotics of velocities are obtained. This paper is mainly a generalization of \citep{BizyaevMamaev2023} to the case of $N$ platforms.

\section{Mathematical model}

{\bf Design and main assumptions.} Consider a wheeled vehicle that moves on a horizontal
plane and consists of $N+1$ coupled platforms (see Fig. \ref{pic0}). The platforms are connected
to each other by means of hinges in such a way that they can freely (i.e., without friction) rotate about the vertical
axis passing through the point of articulation and perpendicular to the
plane of motion.

We make some natural assumptions concerning the design of each of the platforms.	
\begin{itemize}
	\item[{\bf A1.}]  A pair of wheels are rigidly attached to each platform
	and roll without slipping on the supporting plane.
	\item[{\bf A2.}] The center of mass of each wheel pair is at the center of the axis
	connecting the wheels.
	\item[{\bf A3.}]  The center of mass of the platform lies on a straight line which is
	perpendicular to the axis of the wheel pair and passes through the center of mass of the wheel
	pair.
\end{itemize}
We single out one of the outermost platform and assume that several conditions are fulfilled for it.
\begin{itemize}
	\item[{\bf B1.}]  The point of attachment of the next platform  coincides with the center of mass of
	the wheel pair.
	\item[{\bf B2.}]   The platform  has a balanced rotor with a vertical
	axis of rotation. Its angular velocity  is a given periodic function of time.
\end{itemize}
In what follows, we will call this platform the {\it Chaplygin sleigh} because, in the absence of
other platforms ($N=0$)  and a rotor, this system reduces to the well-known problem of the
motion of the Chaplygin sleigh \citep{borisov2015hadamard}.
We will call the other $N$ platforms a {\it trailer}, for which the following
restriction is fulfilled.
\begin{itemize}
	\item[{\bf C1.}] For all platforms of the trailer, the center of mass of the wheel pair
	(which is the point of attachment of the wheel pair to the platform) lies in the middle of the
	segment joining the hinges.
\end{itemize}

The wheeled vehicle consisting of $N+1$ platforms for which assumptions {\bf A} and
{\bf B1} are satisfied, but such that the point of attachment of the wheel pair of each
platform (except for the first and the last) coincides with the point of articulation with
another platform was discussed in~\citep{Naranjo2015}.

\begin{figure*}[h!]
	\centering
	\includegraphics[scale=0.95]{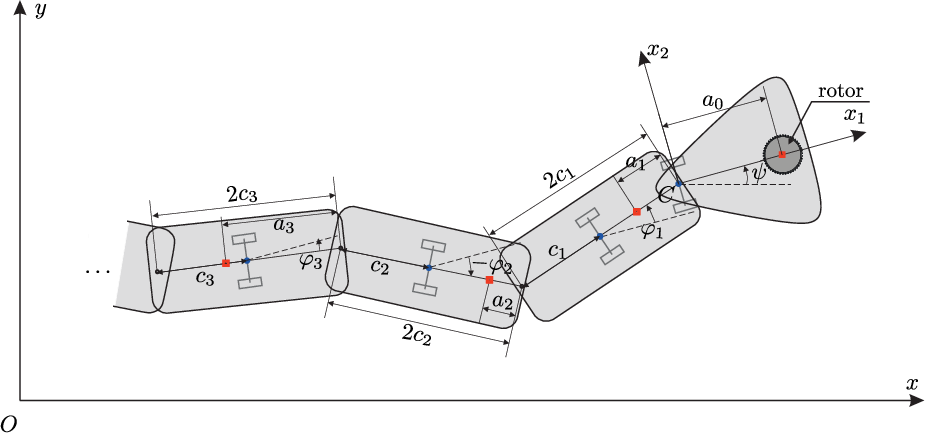}
	\caption{Diagrammatic representation of an $N$-link vehicle on a plane.}
	\label{pic0}
\end{figure*}

{\bf Configuration variables and notation.} To describe the motion of the system, we
introduce generalized coordinates parameterizing the  configuration space. To do so, we
define two coordinate systems:
\begin{itemize}
	\item[---] a \textit{fixed} coordinate system $O x y$ whose axes are fixed and
	lie in the supporting plane;
	\item[---] a \textit{moving} coordinate system $C x_1 x_2$ rigidly attached to the
	leading platform as shown in Fig.\ref{pic0}. Point $C$ coincides with the center of mass of
	the wheel pair, and the axis $Cx_1$ passes through the center of mass of the platform.
\end{itemize}
Let $\boldsymbol{r}_C = (x, y)$ be the radius vector of the center of mass of the wheel
pair of the Chaplygin sleigh (point~$C$) in the coordinate system $Oxy$. We will specify
the orientation of the Chaplygin
sleigh by the angle $\psi$ between the axes $Ox$ and $Cx_1$ and assume that the counterclockwise
direction from the axis $Ox$ is positive. In addition, we let $\varphi_i$ denote the angle between
the $i$th platform of the trailer and the Chaplygin sleigh that is measured from the axis
of the sleigh in the counterclockwise direction (see Fig.~\ref{pic0}). We note that the angles between the platforms are passive, i.e., they are not defined by predetermined functions of time or generalized coordinates.

Thus, the configuration space $\mathcal{N}$ for the system is
\begin{gather*}
\mathcal{N}=\{ \boldsymbol{q}=(x, y, \psi, \varphi_1, \varphi_2, \ldots, \varphi_{N}) \} \approx \mathbb{R}^2 \times\mathbb{T}^{N+1}.
\end{gather*}
To describe an ($N+1$)-link vehicle, we need to define, in addition to the generalized
coordinates $\boldsymbol{q}$, its mass-geometric characteristics, which are presented in
Table \ref{tab1}.

\begin{table}[h!]
	\centering
	\caption{Mass-geometric characteristics of an ($N+1$)-link vehicle}
	\label{tab1}
	\begin{tabular}{|c|l|}
		\hline
		Notation & Description\\
		\hline
		$m_i$ & Mass of the $i$th platform \\
		&  ($i = 0, \ldots, N$)\\
		\hline
		$I_i$ & Central moment of inertia of the $i$th platform \\
		& ($i = 0, \ldots, N$)\\
		\hline
		$a_0$ & Distance from the center of mass of the sleigh \\
		&  to the center of mass
		of the wheel pair \\
		\hline
		$c_i$ & Distance from the hinge of the $i$th platform \\
		& to the point of
		attachment of the wheel\tabularnewline &   pair  ($i = 1, \ldots, N$)\\
		\hline
		$a_i$ & Distance from the center of mass of the \\
		&  $i$th platform, measured from the right
		\tabularnewline &  hinge ($i = 1, \ldots, N$)\\
		\hline
	\end{tabular}
\end{table}

{\bf Nonholonomic constraints.} In \citep{borisov2015hadamard} it was shown that for the
system considered the conditions for the wheels to roll without slipping can be replaced with
the Chaplygin constraints\footnote{That is, we assume that the velocity of the platform at
	the center of mass of the wheel pair in the direction perpendicular to the
	plane of the wheels is zero.} and that the degrees of freedom describing
the turning angles of the wheels can be ignored. Therefore, in what follows we replace
each wheel pair with a weightless knife edge (skate) located at its center of mass.

Let us define the unit vectors directed along the normal and tangent to the plane of the
knife edge. For the sleigh they have the form
\begin{gather*}
\begin{gathered}
\boldsymbol{n}_0 = \big(-\sin \psi,\ \cos \psi \big), \quad \boldsymbol{\tau}_0 = \big(\cos \psi,\ \sin \psi \big),
\end{gathered}
\end{gather*}
and for the $i$th platform of the trailer they are given by the relations
\begin{flalign*}
\boldsymbol{n}_i &= \big(-\sin (\psi + \varphi_i),\ \cos(\psi + \varphi_i) \big), \quad
\boldsymbol{\tau}_i = \big(\cos (\psi + \varphi_i),\ \sin (\psi + \varphi_i)\big).
\end{flalign*}
Then, in the chosen variables, the constraint equations can be represented as
\begin{equation}
\label{eq_f}
\begin{aligned}
\dot{\boldsymbol{r}}_C \cdot \boldsymbol{n}_0 =& -\dot{x} \sin \psi + \dot{y} \cos \psi=0, \\
\dot{\boldsymbol{r}}_{C_i} \cdot \boldsymbol{n}_i =& -\dot{x} \sin(\psi + \varphi_i) + \dot{y} \cos(\psi +  \varphi_i) - 
&2\sum_{j=1}^{i-1} c_j (\dot{\psi} +\dot{\varphi}_j) \cos (\varphi_i - \varphi_j) - c_i (\dot{\psi} + \dot{\varphi}_i) = 0,
\end{aligned}
\end{equation}
where $\boldsymbol{a} \cdot \boldsymbol{b}$ denotes the scalar  product of the vectors~$\boldsymbol{a}$~and~$\boldsymbol{b}$.
Here $\boldsymbol{r}_{C_i} = \boldsymbol{r}_C -  2\sum\limits_{j=1}^{i-1} c_j \boldsymbol{\tau}_j - c_i \boldsymbol{\tau}_i$ is the radius vector of the point of attachment of the skate (wheel pair)
of the $i$th platform.

{\bf Kinetic energy.} The total kinetic energy of the system is composed of
the kinetic energies of every single platform:
\begin{equation}
\label{eq.L}
\begin{gathered}
T = T_0 + \sum_{i=1}^N T_i.
\end{gathered}
\end{equation}
Here the kinetic energy of the sleigh, $T_0$, and the kinetic energy of the $i$th platform, $T_i$,
have the form
\begin{gather}
T_0 = \frac{m_0}{2}\big| \boldsymbol{\dot{r}}_C + a_0\dot{\psi} \boldsymbol{n}_0 \big|^2 + \frac{I_0}{2}\dot{\psi}^2 + k(t)\dot{\psi},\label{eq.T_0}\\
T_i = \frac{m_i}{2}\big| \boldsymbol{\dot{r}}_C - a_i(\dot{\psi} + \dot{\varphi}_i) \boldsymbol{n}_i- 2\sum_{j=1}^{i-1}c_j(\dot{\psi} + \dot{\varphi}_j)\boldsymbol{n}_j \big|^2 + \frac{I_i}{2}(\dot{\psi} + \dot{\varphi}_i)^2,\notag	
\end{gather}
where $k(t)$ is the angular momentum generated by rotating the rotor.

\begin{rmk}
	The expression for the kinetic energy $T_0$ can be written as
	\begin{gather}\label{T_0}
		T_0 =  \frac{m_0}{2}\big| \boldsymbol{\dot{r}}_C + a_0\dot{\psi} \boldsymbol{n}_0 \big|^2 + \frac{I_p}{2}\dot{\psi}^2 + \frac{I_r}{2}(\dot{\psi} + \dot{\psi}_r)^2,
	\end{gather}
	where $I_p$ and $I_r$ are the moments of inertia of the platform and the rotor, and $\psi_r = \psi_r(t)$ is the turning angle of the rotor. We write equation \eqref{T_0} as
	\begin{gather}
	T_0 = \frac{m_0}{2}\big| \boldsymbol{\dot{r}}_C + a_0\dot{\psi} \boldsymbol{n}_0 \big|^2 + \frac{I_p + I_r}{2}\dot{\psi}^2 + I_r\dot{\psi} \dot{\psi}_r + \frac{I_r}{2}\dot{\psi}_r^2.
	\end{gather}
	The last term depends only on time and hence does not appear in the Lagrange-Euler equations~\eqref{eq_basis}, so that it can be excluded from consideration. Introducing the notation $I_0 = I_p + I_r$, $k(t) = I_r \dot{\psi}_r$, we obtain an expression for the kinetic energy \eqref{eq.T_0}.
\end{rmk}

\textbf{Equations of motion in quasi-velocities.}
Let $\boldsymbol{v} = (v_1, v_2)$ be the translational velocity of point $C$
referred to the coordinate system $C x_1 x_2$, and  let $\omega$ be the absolute angular
velocity of the sleigh. The relation of the quasi-velocities~$\boldsymbol{v}$ and $\omega$
to the generalized velocities
of the sleigh is given by
\begin{gather*}
v_1 =  \dot{x}\cos\psi + \dot{y}\sin\psi, \quad
v_2 =  -\dot{x}\sin\psi + \dot{y}\cos\psi, \quad
\omega = \dot{\psi}.
\end{gather*}
In addition, we introduce quasi-velocities $\boldsymbol{w} = (w_1, w_2, \ldots, w_N)$,
where  $w_{i}  = \dot{\boldsymbol{r}}_{C_i} \cdot \boldsymbol{n}_i$, $i=1,\ldots, N$.
In the chosen quasi-velocities the constraint equations \eqref{eq_f} have the following
simple form:
\begin{gather}\label{eq_fw}
v_2 = 0, \quad w_i = 0,  \quad i = 1,\ldots, N.
\end{gather}

The expression for the generalized velocities $\dot{\boldsymbol{q}}$ in terms of
the quasi-velocities $\boldsymbol{u} = (\boldsymbol{v}, \omega, \boldsymbol{w})$ can be
represented as
\begin{gather*}
\boldsymbol{\dot{q}} = v_1 \boldsymbol{\sigma}_v  + v_2 \boldsymbol{\nu}_v + \omega \boldsymbol{\sigma}_\omega  + \sum_{i=1}^{N}w_i\boldsymbol{\rho}_i,
\end{gather*}
where $\boldsymbol{\sigma}_v$, $\boldsymbol{\nu}_v$, $\boldsymbol{\sigma}_\omega$, $\boldsymbol{\rho}_i$
are vector fields in a configuration space $\mathcal{N}$ of the form
$$
\begin{aligned}
\boldsymbol{\sigma}_\omega =& \frac{\partial }{\partial \psi} -  \sum_{i=1}^N \frac{\partial }{\partial \varphi_i}, \quad
\boldsymbol{\sigma}_v= \cos\psi\frac{\partial }{\partial x} + \sin\psi\frac{\partial }{\partial y}  + \sum_{i=1}^N \frac{(-1)^{i}}{c_i}  \sin\theta_i \frac{\partial }{\partial \varphi_i}, \\
\boldsymbol{\nu}_v =& -\sin\psi\frac{\partial }{\partial x} + \cos\psi\frac{\partial }{\partial y}  - \sum_{i=1}^N \frac{(-1)^{i}}{c_i}  \cos\theta_i \frac{\partial }{\partial \varphi_i}, \\
\boldsymbol{\rho}_i = &-\frac{1}{c_{i}}\frac{\partial}{\partial \varphi_i} + \frac{2}{c_{i+1}}\cos(\varphi_i - \varphi_{i+1})\frac{\partial }{\partial \varphi_{i+1} } +
2\sum_{k=i+2}^{N}\frac{(-1)^{k+i+1}}{c_k}\cos\big( \theta_k -\theta_i \big)\frac{\partial}{\partial \varphi_k},
\quad i=1,\dots N-1, \\
\boldsymbol{\rho}_N =& -\frac{1}{c_{N}}\frac{\partial}{\partial \varphi_N}, 	
\end{aligned}
$$
where we have introduced the notation
\begin{equation}
\label{eq_theta_i}
\theta_i =  (-1)^{i+1}\varphi_i + 2\sum_{j=1}^{i-1}(-1)^{j+1} \varphi_j.
\end{equation}

To derive the equations of motion for the sleigh (taking the nonholonomic constraints into account) in the chosen quasi-velocities, it suffices to calculate
the commutator of vector fields $\boldsymbol{\sigma}_v$ and $\boldsymbol{\sigma}_\omega$
(for details, see \citep{borisov2015symmetries}), for which the following relation holds:
\begin{equation}
\label{eq_km}
[\boldsymbol{\sigma}_v, \boldsymbol{\sigma}_\omega]= - \boldsymbol{\nu}_v,
\end{equation}
where $[\cdot\, , \, \cdot]$ denotes Lie brackets. Then the equations of motion can be written as
\begin{equation}
\label{eq_basis}
\begin{gathered}
\dfrac{d}{dt}\left( \dfrac{\partial T^{\ast }}{\partial v_{1}}\right) - \boldsymbol{\sigma}_{v}(  T^{\ast }) = \left( \dfrac{\partial T}{\partial v_{2}}\right)^{\ast }\omega, \\
\dfrac{d}{dt}\left( \dfrac{\partial T^{\ast }}{\partial \omega }\right) - \boldsymbol{\sigma}_{\omega }(  T^{\ast } ) = -\left( \dfrac{\partial T}{\partial v_{2}}\right)^{\ast }v_{1},
\end{gathered}
\end{equation}
where $T$ is the kinetic energy \eqref{eq.L} written in quasi-velocities
$\boldsymbol{u} = (\boldsymbol{v}, \omega, \boldsymbol{w})$ and  $( \ )^{\ast }$ denotes substitution of the constraints \eqref{eq_fw}.

It can be seen from the expression \eqref{eq_km} that
the commutator has no vector fields $\boldsymbol{\rho}_i$,
therefore, equations \eqref{eq_basis} contain no terms of the form $\dfrac{\partial T}{\partial w_i}$.
Hence, to derive the equations of motion, it suffices to represent the kinetic energy as
$$
T = T^* + Cv_2 + \dots,
$$
where the dots denote terms depending on $\boldsymbol{w}$ and quadratic in $v_2$.
The expressions for the kinetic energy $T^*$ and the coefficient $C$ have the form
\begin{equation}
\begin{aligned}
\label{eq_TC}
&T^* = \frac{1}{2}\Phi_0(\boldsymbol{\varphi})v_1^2 + \frac{J}{2}\omega^2 + k(t)\omega, \quad
C = b\omega - \frac{v_1}{2}\sum_{i=1}^N\mu_i\sin(2 \theta_i), \\
&\Phi_0(\boldsymbol{\varphi}) = m + \sum_{i=1}^N\mu_i \sin^2 \theta_i>0, \quad  \boldsymbol{\varphi}=(\varphi_1, \dots, \varphi_N),
\end{aligned}
\end{equation}
where the following parameters are introduced:
\begin{gather}
\begin{gathered}\label{eq_J_m}
J = I_0 + m_0a_0^2, \quad m=\sum_{i=0}^Nm_i, \quad b=m_0a_0, \quad \mu_i=\frac{1}{c_i^2}\big(I_i + m_ia_i(a_i - 2c_i)\big).
\end{gathered}
\end{gather}	

Substituting the expressions \eqref{eq_TC} into \eqref{eq_basis}, we obtain equations
governing the evolution of $(v_1, \omega)$:
\begin{equation}
\label{eq_dvdw}
\begin{gathered}
\Phi_0(\boldsymbol{\varphi})\dot{v}_1 = b\omega^2 + \Phi_1(\boldsymbol{\varphi}) v_1^2 + \Phi_2(\boldsymbol{\varphi})\omega v_1, \\
J\dot{\omega} = {-}b\omega v_1 - \dot{k}, \\
\Phi_1(\boldsymbol{\varphi}) = \sum_{i=1}^N\mu_i\sin(2\theta_i)\left[ \frac{\sin \theta_i}{2c_i} + \sum_{j=1}^{i-1} \frac{\sin \theta_j}{c_j}  \right], \quad
\Phi_2(\boldsymbol{\varphi}) = {\frac{1}{2}}\sum_{i=1}^N\mu_i\sin(2\theta_i).
\end{gathered}
\end{equation}	
This system needs to be supplemented with equations governing the evolution of the angles
between the platforms:
\begin{equation}
\label{eq_dphi}
\dot{\varphi}_i=(-1)^{i}\frac{v_1}{c_i}\sin \theta_i - \omega, \quad i=1,\dots, N.
\end{equation}	
Equations \eqref{eq_dvdw} and \eqref{eq_dphi} form a reduced system.
To recover the motion of the Chaplygin sleigh in the fixed coordinate system from the known
solution of the reduced system, it is necessary to supplement it with the kinematic relations
\begin{gather}
\label{eq_kin}
\dot{x} = v_1 \cos \psi, \quad \dot{y} = v_1 \sin \psi, \quad \dot{\psi} = \omega.
\end{gather}

For the system we consider here, the following case is singled out:
\begin{equation}
\label{eq_mu}
\mu_1=0, \quad \dots, \quad \mu_N=0,
\end{equation}
in which the angles $\varphi_1, \dots, \varphi_N$ do not appear explicitly in the
equations of motion \eqref{eq_dvdw}.
As a result, the equations of motion
governing the evolution of the velocities, which coincide with the equations for the
Chaplygin sleigh, decouple from the general system.

\begin{rmk}
	Condition \eqref{eq_mu} implies that $I_i =  m_ia_i(2c_i - a_i)$.  Such a moment of inertia corresponds, for example, to the case where the mass distribution of the platform reduces to two point masses located as shown in Fig. \ref{fig0_} ($A_i\xi_i\eta_i$ is  the coordinate system on the $i$th platform that is attached to its center of mass, $d = \sqrt{a_i(2c_i - a_i)}$).	
	\begin{figure}[h]
		\centering
		\includegraphics{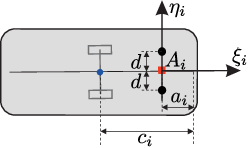}
		\caption{A schematic representation of the mass distribution on the $i$th platform for which $\mu_i=0$.}\label{fig0_}
	\end{figure}
\end{rmk}

\begin{rmk}
	Condition \eqref{eq_mu} was obtained in \citep{Mikishanina2021} for a similar wheeled
	vehicle consisting of two platforms.
\end{rmk}

	\section{The inertial motion and stability of partial solutions}\label{sec21}
Let the rotor be at rest ($k=0$). Then the equations of motion~\eqref{eq_dvdw} and \eqref{eq_dphi}
admit the energy integral
\begin{equation}
\label{eqE}
E = T^* = \frac{\Phi_0(\boldsymbol{\varphi})}{2}v_1^2 + \frac{J}{2}\omega^2 = \mathrm{const}.
\end{equation}

Let us fix the level set of the energy integral $E(v_1, \omega, \boldsymbol{\varphi})=h$
and introduce a new angle variable~$\vartheta \in[0, 2\pi)$:
$$
v_1=\sqrt{\frac{2h}{\Phi_0(\boldsymbol{\varphi})}}\cos\vartheta, \quad  \omega = \sqrt{\frac{2h}{J}}\sin\vartheta.
$$
Then, after rescaling time as $dt = \dfrac{\sqrt{\Phi_0(\boldsymbol{\varphi})}}{\sqrt{2h}}d\tau$,
the equation governing the evolution of the angle variable introduced above decouples from the general system:
\begin{equation}
\label{eq_vartheta}
\vartheta' = -\frac{b}{J}\sin\vartheta,
\end{equation}
where the prime denotes the derivative with respect to the new time $\tau$.
The equations for the angles defining the orientation of the platforms of the trailer
can be represented as
\begin{equation}
\label{eq_vartheta2}
\varphi_i' = \frac{(-1)^i}{c_i} \cos\vartheta \sin\theta_i - \frac{\sqrt{ \Phi_0(\boldsymbol{\varphi})}}{\sqrt{J}}\sin\vartheta, \quad i=1,\dots N.
\end{equation}

The system \eqref{eq_vartheta}-\eqref{eq_vartheta2} possesses isolated $2^{N+1}$ fixed
points which correspond to the motion of all platforms in the same straight line.
We divide them into two families, depending on the sign of the translational velocity $v_1$:
$$
\begin{gathered}
\Sigma_{+} = \big\{ \vartheta=0, \quad  \sin \theta_i=0, \quad i= 1,\dots N\big\};\\
\Sigma_{-} =\big\{ \vartheta=\pi, \quad  \sin \theta_i=0, \quad i= 1,\dots N\big\}.
\end{gathered}
$$
Thus, the family $\Sigma_{+}$ corresponds to motion where the Chaplygin sleigh
moves in the positive direction of the axis $Cx_1$, and the family $\Sigma_{-}$ corresponds
to motion of the sleigh in the opposite direction. In both families $\Sigma_\pm$
the fixed points differ, according to \eqref{eq_theta_i}, in the magnitude of angle
$\varphi_i$,
which takes the values 0 or $\pi$. If $\varphi_i = 0$, then the $i$th platform is aligned
relative to the preceding platform, and if $\varphi_i = \pi$, the platform overlaps with it.

To analyze the stability of the above-mentioned fixed
points, we introduce the following quantities:
$$
\sigma_0 = \cos\vartheta, \quad \sigma_i = \cos\theta_i, \quad i=1,\dots N,
$$
which for the families $\Sigma_\pm$ can take the values $+1$ or $-1$. Then the linearization
matrix of the system \eqref{eq_vartheta}-\eqref{eq_vartheta2} in a neighborhood of fixed points can be
represented as
\begin{gather}\small
\begin{gathered}
{\bf A} = \begin{pmatrix}
-\dfrac{b}{J}\sigma_0 &  0 & 0 &  \ldots & 0\\
- \dfrac{\sqrt{m}}{\sqrt{J}} \sigma_0 & -\dfrac{\sigma_0}{c_1} \sigma_1 & 0 & \ldots & 0\\
- \dfrac{\sqrt{m}}{\sqrt{J}} \sigma_0& \dfrac{2 \sigma_0}{c_2} \sigma_2 & -\dfrac{\sigma_0}{c_2} \sigma_2 & \ldots & 0\\
\hdotsfor{5}\\
- \dfrac{\sqrt{m}}{\sqrt{J}} \sigma_0 & (-1)^{N+2}\dfrac{2 \sigma_0}{c_N} \sigma_N & (-1)^{N+3}\dfrac{2 \sigma_0}{c_N} \sigma_N & \ldots & -\dfrac{\sigma_0}{c_N} \sigma_N
\end{pmatrix}.
\end{gathered}
\end{gather}
As can be seen, ${\bf A}$ is a triangular matrix, and hence all its eigenvalues lie on the main diagonal.
Thus, the following proposition holds.
\begin{pro}\label{pro1}
	The fixed points $\vartheta = 0$, $\boldsymbol{\varphi}=0$ and $\vartheta = \pi$, $\boldsymbol{\varphi}=0$
	are a stable and an unstable node, respectively. The other fixed points are of saddle type
	and hence unstable.
\end{pro}

The proof follows from direct substitution of the fixed points into the
linearization matrix ${\bf A}$.

Next, we consider the system's trajectories different from the fixed points.
We first note that the
system \eqref{eq_vartheta}-\eqref{eq_vartheta2}
possesses two invariant manifolds
\begin{equation}
\label{eqInvariant2}
\begin{gathered}
\mathcal{M}_+^{N-1}=\{ (\vartheta, \boldsymbol{\varphi}) \ | \  \vartheta=0 \}, \quad
\mathcal{M}_-^{N-1}=\{ (\vartheta, \boldsymbol{\varphi}) \ | \  \vartheta=\pi \},
\end{gathered}
\end{equation}
on which the above-mentioned fixed points $\Sigma_\pm \subset \mathcal{M}_\pm^{N-1}$ lie.
On these manifolds, the evolution of the remaining angles has the form
\begin{equation}
\label{eqInvariant}
\varphi_i'=\pm\frac{(-1)^i}{c_i}\sin\theta_i, \quad i=1,\dots N,
\end{equation}
where the upper sign corresponds to $\mathcal{M}_+^{N-1}$, and the lower sign, to
$\mathcal{M}_-^{N-1}$.
According to equations \eqref{eqInvariant}, the trajectories on these manifolds
are identical up to time inversion $\tau \to -\tau$.

As can be seen from equation \eqref{eq_vartheta}, the trajectories that do not lie on
\eqref{eqInvariant2} approach asymptotically $\mathcal{M}_+^{N-1}$. Numerical experiments
show that all trajectories lying on this manifold, in their turn, asymptotically approach
a stable node, except for saddle equilibrium points and the corresponding separatrices.

We illustrate this by a three-link vehicle $N = 2$.
If we denote the fixed points as
$$
\begin{gathered}
\Sigma_+^{(\phi_1, \phi_2)} = \{ \vartheta=0, \  \varphi_1=\phi_1, \varphi_2=\phi_2 \}, \\
\Sigma_-^{(\phi_1, \phi_2)} = \{ \vartheta=\pi, \  \varphi_1=\phi_1, \varphi_2=\phi_2 \},
\end{gathered}
$$
then there exist 8 fixed points:
\begin{enumerate}
	\item a stable node $\Sigma_+^{(0,0)}$ and an unstable  $\Sigma_-^{(0,0)}$ node;
	\item saddle points $\Sigma_\pm^{(0,\pi)}$, $\Sigma_\pm^{(\pi,0)}$, $\Sigma_\pm^{(\pi,\pi)}$.
\end{enumerate}

\begin{figure}[h]
	\centering
	\includegraphics{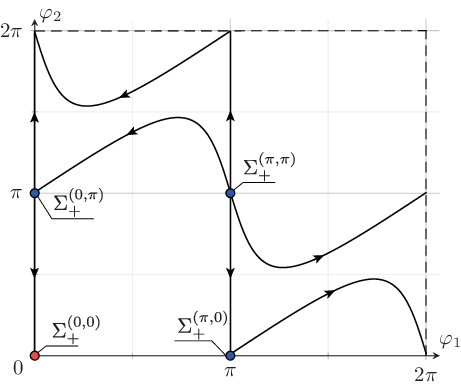}
	\caption{Phase portrait of the system \eqref{eqInvariant2} in the case of the
		manifold $\mathcal{M}_+^{2}$ for fixed $c_1=1$, $c_2=1.5$.}\label{fig01}
\end{figure}

\begin{figure*}[h]
	\centering
	\includegraphics[scale=0.85]{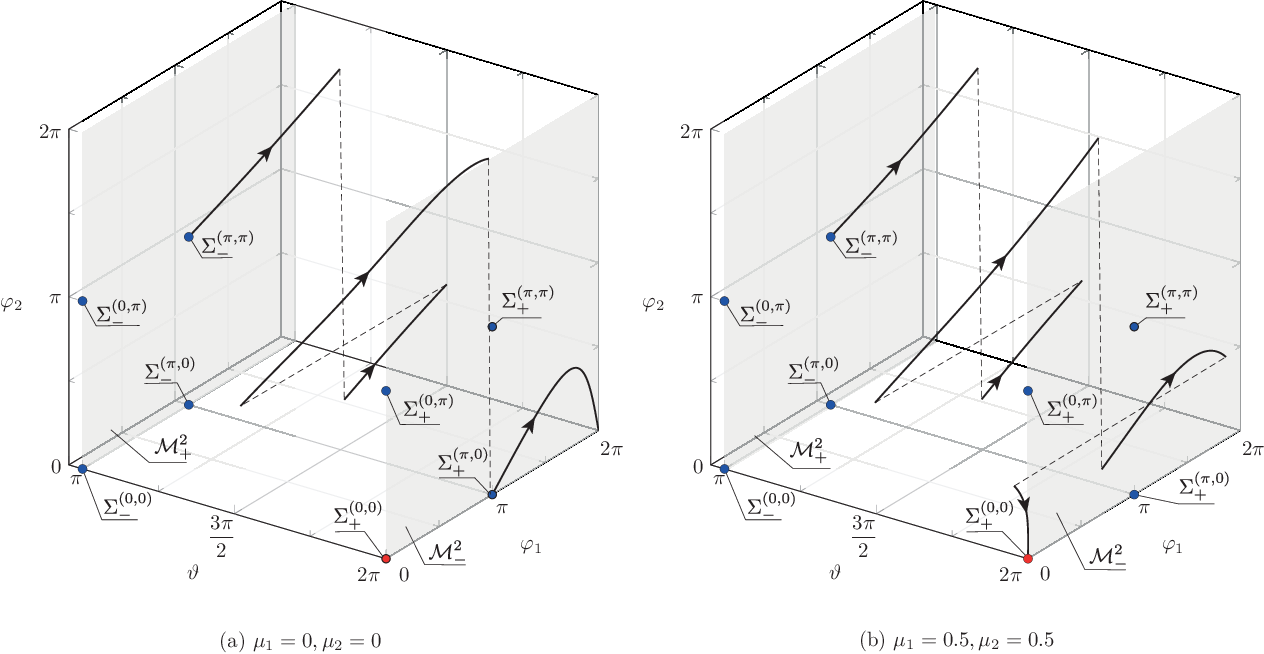}
	\caption{An example of the separatrix of the fixed point $\Sigma_-^{(\pi,\pi)}$ in the
		space $(\vartheta, \varphi_1, \varphi_2)$ for the fixed parameters  $c_1=1$, $c_2=1.5$, $b=0.5$, $m=3$, $J=1$ and different $\mu_1$ and $\mu_2$.
		The initial conditions (a) $\vartheta(0)=3.143592654$,  $\varphi_1(0)=3.143902055$, $\varphi_2(0)=3.147201199$, (b) $\vartheta(0)=3.141602654$,  $\varphi_1(0)=3.141604201$, $\varphi_2(0)=3.141620697$.}
	\label{fig02}
\end{figure*}

A typical phase portrait on $\mathcal{M}_+^{2}$ is shown in Fig. \ref{fig01}, where one can clearly see the separatrices joining the equilibrium points
$\Sigma_+^{(\pi,\pi)}$ and $\Sigma_+^{(\pi,0)}$, as well as $\Sigma_+^{(\pi,\pi)}$ and
$\Sigma_+^{(0, \pi)}$. Consequently, the trajectory in this portrait with initial conditions
in a neighborhood of $\Sigma_+^{(\pi,\pi)}$ is not ``attracted'' immediately to a stable
equilibrium point, but evolves at first into the neighborhood of the equilibrium point
$\Sigma_+^{(\pi,0)}$ or $\Sigma_+^{(0,\pi)}$.
The same can be valid for the trajectories that do not lie on $\mathcal{M}_+^{2}$.
An example of such a trajectory  in case \eqref{eq_mu} is given in Fig. \ref{fig02}a.
Figure \ref{fig02}b shows an example of a trajectory which also emanates from the neighborhood
of $\Sigma_-^{(\pi,\pi)}$, but which is attracted ``immediately'' to the equilibrium point
$\Sigma_-^{(0,0)}$ without entering into the neighborhoods of other fixed points.

The trajectories of the point of attachment of the wheel pairs (skates) to the Chaplygin
sleigh and each platform of the trailer during motion along the trajectory in Fig. \ref{fig02}
are shown in Fig. \ref{fig03} and Fig. \ref{fig04}.  In these figures,
one can clearly see cusp points which correspond to the U-turn of the corresponding platform.

\begin{figure}[h]
	\centering
	\includegraphics[scale=0.85]{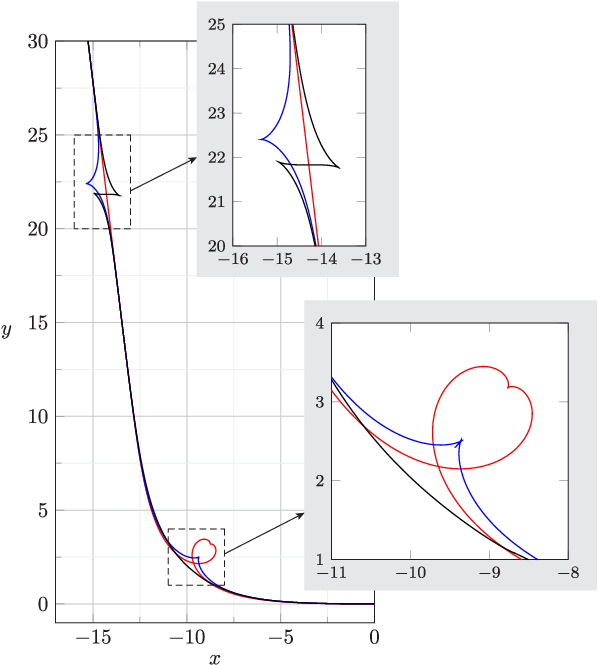}
	\caption{The trajectory of the points of attachment of the wheel pairs in the
		three-link vehicle for the trajectory in Fig. \ref{fig02}a and the initial conditions
		$\psi(0)=0$, $x(0)=0$, $y(0)=0$. The red line corresponds to the Chaplygin sleigh, and the
		blue and black lines correspond to the first and the second platform of the trailer, respectively. }
	\label{fig03}
\end{figure}

\begin{figure}[h]
	\centering
	\includegraphics[scale=0.85]{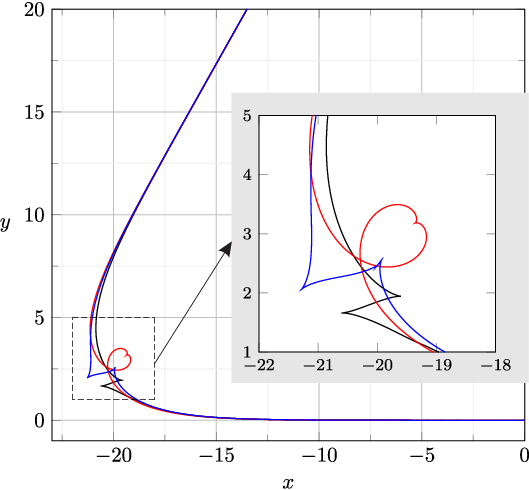}
	\caption{The trajectory of the points of attachment of the wheel pairs in the three-link
		vehicle for the trajectory in Fig. \ref{fig02}b and the initial conditions $\psi(0)=0$, $x(0)=0$, $y(0)=0$.
		The red line corresponds to the Chaplygin sleigh, and the blue and black lines correspond to
		the first and the second platform of the trailer, respectively. }
	\label{fig04}
\end{figure}

\section{Acceleration with a rotor added}
In the preceding section, it was shown that in the absence of  the angular momentum of the
rotor ($k=0$) the phase space of the system \eqref{eq_dvdw}-\eqref{eq_dphi} is foliated
on the level surface of the energy integral \eqref{eqE}.
On each level surface there is an asymptotically stable fixed point of node type for
which
$\boldsymbol{\varphi}=0$, $v_1={\rm const}$, $\omega=0$. This gives rise to a one-parameter
family of asymptotically stable fixed points in the phase space of the initial system.

In this section, we consider what influence
the angular momentum $k(t)$ has on the motion of the ($N+1$)-link vehicle.
We will specify $k(t)$ in the form of a periodic function of time with period $P$. In this
case, the energy is not conserved. However, from a physical point of view it is clear that,
for sufficiently large values of velocities, the contribution of the angular momentum
turns out to be a small quantity. Hence, the trajectories, with initial conditions in a
small neighborhood of the one-parameter family of equilibrium points with large values of $v_1$,
must stay in the same neighborhood for a fairly long time. In the general case  $v_1$
can both increase and decrease.

To analyze the equations of motion \eqref{eq_dvdw} and \eqref{eq_dphi}, it is more convenient to
transform from the variables $\varphi_i$ to the variables $\theta_i$ according to
relation \eqref{eq_theta_i}. This transformation can be represented as
\begin{gather*}
\boldsymbol{\theta} = {\bf B} \boldsymbol{\varphi}^{T}, \quad {\bf B} = \begin{pmatrix}
1 & 0 & 0 & \ldots & 0\\
2 & -1 & 0 & \ldots & 0\\
2 & -2 & 1 & \ldots & 0\\
\hdotsfor{5}\\
2 & -2 & 2 & \ldots & (-1)^{N+1}\\
\end{pmatrix},
\end{gather*}
where ${\bf B}$ is a unimodular matrix. Then the system under study takes the form
\begin{gather}\label{eq_dvdwk}
\begin{gathered}
\Phi_0(\boldsymbol{\theta})\dot{v}_1 = b\omega^2 + \Phi_1(\boldsymbol{\theta}) v_1^2 + \Phi_2(\boldsymbol{\theta})\omega v_1, \quad
J\dot{\omega} = {-}b\omega v_1 - \dot{k},\\
\dot{\theta}_i= -\frac{v_1}{c_i}\sin \theta_i - 2 \sum_{j=1}^{i - 1} \frac{v_1}{c_j}\sin \theta_j - \omega, \quad i=1,\dots N.
\end{gathered}
\end{gather}
\begin{rmk}
	The equations for $\theta_i$, $i=1,\dots, N$ in \eqref{eq_dvdwk}
	have been obtained using the identity
	\begin{gather}\label{prop1}
	(-1)^i + 2 \sum_{j=1}^{i-1} (-1)^j = -1.
	\end{gather}
\end{rmk}

To analyze the system \eqref{eq_dvdwk}, we use the approach proposed in \citep{BizyaevMamaev2023},
where a criterion for the onset of accelerating trajectories is found for the Roller Racer
with a periodically changing angular momentum. Both systems share the feature that, in the
absence of angular momentum, a one-parameter family of 
equilibrium points exists in the phase space, and these 
systems differ in the dimension of the phase space. As a result, the following proposition holds
for the system \eqref{eq_dvdwk}.

\begin{pro}\label{pro2}
	There exist $\varepsilon>0$ small enough and constant $\widetilde{\omega}$, $\widetilde{\theta}_i$
	such that, for the trajectory of the reduced system \eqref{eq_dvdwk} with the initial
	conditions $\dfrac{1}{\varepsilon}<v_1$, $|\omega|<\widetilde{\omega}\varepsilon^2$, $|\theta_i|<\widetilde{\theta}_i\varepsilon^2$,
	we have $v_1\to +\infty$, $\omega \to 0$, $\theta_i \to 0$ in the form
	\begin{gather}\label{eq.appr}
	\begin{gathered}
	v_1(t) =  \Delta^{1/3} t^{1/3} + o(t^{1/3}), \quad \omega(t) = -\frac{\dot{k}}{b \Delta^{1/3}} t^{-1/3} + o(t^{-1/3}),\\	
	\theta_i(t) = (-1)^{i + 1} \frac{c_i \dot{k}}{b} \Delta^{-2/3} t^{-2/3}  + o(t^{-2/3}), \quad i = 1, 2, \ldots, N,\\
	\Delta = \frac{3 \langle \dot{k}^2\rangle}{b m}, \quad \langle \dot{k}^2\rangle = \frac{1}{P} \int\limits_0^P \dot{k}^2 dt.
	\end{gathered}
	\end{gather}
\end{pro}

\begin{proof}
	Let us pass from the velocities $v_1$, $\omega$ to the variables $p$, $q$ as follows:
	\begin{gather}
	p = \frac{1}{v_1}, \quad q = \frac{\omega}{v_1},
	\end{gather}
	and rescale time as $d\tau = v_1 d t = d t / p$.
	In addition, we define the angle variable $\Psi = t \mod P$.
	Then equations \eqref{eq_dvdwk} in the new variables become
	\begin{gather}
	\begin{gathered}\label{eq_pq_tau}
	{p}' = -\frac{p}{\Phi_0(\boldsymbol{\theta})} \bigg(b q^2 + \Phi_1(\boldsymbol{\theta}) + q \Phi_2(\boldsymbol{\theta})\bigg), \\
	{q}' = -\frac{p^2}{J} \bigg( \frac{b q}{p^2} + \frac{dk}{d\Psi} \bigg) - \frac{q}{\Phi_0(\boldsymbol{\theta})} \bigg(b q^2 + \Phi_1(\boldsymbol{\theta}) + q \Phi_2(\boldsymbol{\theta})\bigg),\\
	{\theta}_i' = -\frac{\sin \theta_i}{c_i} - 2 \sum_{j=1}^{i - 1} \frac{\sin \theta_j}{c_j} - q, \quad i=1,\dots, N, \\
	{\Psi}'=p,
	\end{gathered}
	\end{gather}
	where the prime denotes the derivative with respect to the variable $\tau$.
	
	Note that $p=0$ defines the invariant manifold of the system \eqref{eq_pq_tau},
	which has a one-parameter family of fixed points given by the following relations:
	\begin{gather}\label{eq_fps}
	p = 0, \quad q = 0, \quad \theta_i = 0, \quad i = 1,2, \ldots, N,
	\end{gather}
	i.e., it is parameterized by the variable $\Psi$.
	In a neighborhood of this family the system \eqref{eq_pq_tau} can be represented as
	
	\begin{figure*}[h]
		\centering
		\includegraphics[width=0.95\linewidth]{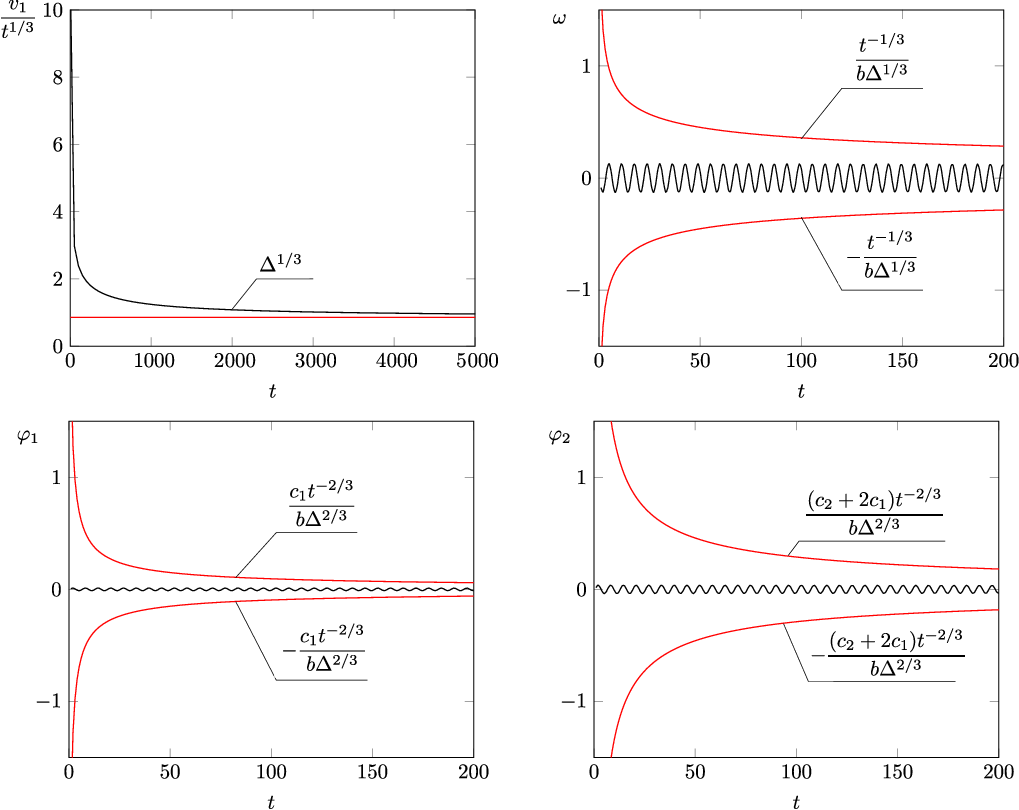}
		\caption{Dependences $t^{-1/3} v_1(t)$ and $\omega(t)$ (above), $\varphi_1(t)$ (below
			on the left), $\varphi_2(t)$ (below on the right)}\label{pic9}
	\end{figure*}
	
	\begin{figure*}[h]
		\centering
		\includegraphics[width=0.95\linewidth]{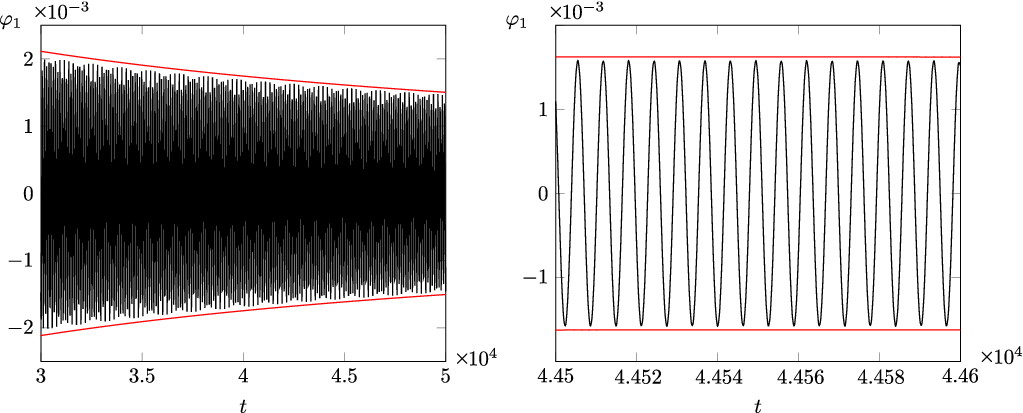}
		\caption{An enlarged fragment of the dependence $\varphi_1(t)$ at large times}\label{pic10}
	\end{figure*}

	\begin{gather}
	\frac{d \boldsymbol{z}}{d \tau} = {\bf C} \boldsymbol{z} + \boldsymbol{F}(\boldsymbol{z}, \Psi), \quad \frac{d\Psi}{d\tau}=p, \qquad \boldsymbol{z} = (p, q, \theta_1, \ldots, \theta_N), \label{dz} \\
	{\bf C} =
	\begin{pmatrix}
	0 & 0 & 0 & 0 & \ldots & 0\\
	0 & -\dfrac{b}{J} & 0 & 0 & \ldots & 0\\
	0 & -1 & -\dfrac{1}{c_1} & 0 & \ldots & 0\\
	0 & -1 & -\dfrac{2}{c_1} & -\dfrac{1}{c_2} & \ldots & 0\\
	\hdotsfor{6}\\
	0 & -1 & -\dfrac{2}{c_1} & -\dfrac{2}{c_2} & \ldots & -\dfrac{1}{c_N}
	\end{pmatrix},\notag
	\end{gather}
	where $\boldsymbol{F}(\boldsymbol{z}, \Psi)$ is a periodic function in $\Psi$ which contains
	$z$ of second and higher degrees. Note that the matrix of the linear part ${\bf C}$ has one
	zero eigenvalue and the other nonzero eigenvalues are negative.

	Since the spectrum of the linear part of the system \eqref{dz} contains only
	zero and negative eigenvalues, we represent $q$, $\theta_1$, $\theta_2$, $\ldots$, $\theta_N$
	as series in powers of $p$:
	\begin{gather}\label{series_q_varphi}
	\begin{gathered}
	q = \alpha_2 p^2 + \alpha_3 p^3 + O(p^4), \\ \theta_1 = \beta^{(1)}_2 p^2 + \beta^{(1)}_3 p^3 + O(p^4), \\
	\theta_2 = \beta^{(2)}_2 p^2 + \beta^{(2)}_3 p^3 + O(p^4), \quad \ldots, \\ \theta_N = \beta^{(N)}_2 p^2 + \beta^{(N)}_3 p^3 + O(p^4).
	\end{gathered}
	\end{gather}
	The expansions of the functions $\Phi_0(\boldsymbol{\theta})$, $\Phi_1(\boldsymbol{\theta})$, $\Phi_2(\boldsymbol{\theta})$ can be represented as
	\begin{gather}\label{series_Phi_varphi}
	\begin{aligned}
	\Phi_0(\boldsymbol{\theta}) = &m + \sum_{i=1}^N \mu_i \big(\beta^{(i)}_2 p^2 + \beta^{(i)}_3 p^3\big)^2 = 
	 m + \gamma_1^{(0)} p^4 + \gamma_2^{(0)} p^5 + \ldots,\\
	\Phi_1(\boldsymbol{\theta}) =& \sum_{i=1}^N 2 \mu_i \big(\beta^{(i)}_2 p^2 + \beta^{(i)}_3 p^3\big) \bigg[ \frac{\beta^{(i)}_2 p^2 + \beta^{(i)}_3 p^3}{2c_i} 
	+ \sum_{j=1}^{i-1} \frac{\beta^{(j)}_2 p^2 + \beta^{(j)}_3 p^3}{c_j}  \bigg] =
	\gamma_1^{(1)} p^4 + \gamma_2^{(1)} p^5 + \ldots,\\
	\Phi_2(\boldsymbol{\theta}) =& {\frac{1}{2}}\sum_{i=1}^N 2 \mu_i \big(\beta^{(i)}_2 p^2 + \beta^{(i)}_3 p^3\big)= \gamma_1^{(2)} p^2 + \gamma_2^{(2)} p^3 + \ldots,
	\end{aligned}
	\end{gather}
	where $\gamma_{1}^{(i)}$ and $\gamma_{2}^{(i)}$ are some constants expressed in terms
	of the expansion coefficients \eqref{series_q_varphi} and the system parameters.
	
	Substituting \eqref{series_q_varphi} and \eqref{series_Phi_varphi} into
	\eqref{eq_pq_tau} and matching the coefficients in front of equal powers of $p$, we obtain
	expressions for the
	expansion coefficients \eqref{series_q_varphi}:
	\begin{gather}\label{coeffs}
	\begin{gathered}
	\alpha_2 = -\frac{1}{b}\frac{dk}{d\psi}, \quad \alpha_3 = 0,\quad
	\beta^{(i)}_2 = (-1)^i c_i \alpha_2 = (-1)^{i + 1} \frac{c_i}{b}\frac{dk}{d\psi}, \\ \beta^{(i)}_3 = 0, \quad i = 1, 2, \ldots, N.
	\end{gathered}
	\end{gather}
	After substituting \eqref{coeffs} into the expansions \eqref{series_Phi_varphi}
	using~\eqref{prop1}, they take the following form:
	\begin{gather*}
	\Phi_0(\boldsymbol{\theta}) = m + \alpha_2^2 p^4 \sum_{i=1}^N \mu_i c_i^2, \quad
	\Phi_1(\boldsymbol{\theta}) = -\alpha_2^2 p^4 \sum_{i=1}^N (-1)^i \mu_i c_i,\quad
	\Phi_2(\boldsymbol{\theta}) = \alpha_2 p^2 \sum_{i=1}^N (-1)^i \mu_i c_i.
	\end{gather*}
	We now substitute these expansions into the first of equations \eqref{eq_pq_tau} and
	divide it by the last equation. Finally, setting $p \ll 1$, after averaging over $\Psi$ and
	passing to the initial time $t$, we obtain
	\begin{gather}
	\dot{p} = -\frac{\langle\dot{k}^2\rangle}{bm}p^4, \quad p(t) = \Delta^{-1/3} t^{-1/3}, \label{eq_p} \\ \Delta = \frac{3 \langle \dot{k}^2\rangle}{b m}, \quad \langle \dot{k}^2\rangle = \frac{1}{P} \int\limits_0^P \dot{k}^2 dt.\notag
	\end{gather}
	
	Substituting \eqref{eq_p} and \eqref{coeffs} into the expansion \eqref{series_q_varphi}
	and transforming to the initial variables $v_1$, $\omega$ and $\theta_i$, we obtain the
	estimates \eqref{eq.appr}.
\end{proof}

The physical meaning of Proposition \ref{pro2} is that under some initial
conditions the Chaplygin sleigh exhibits an unbounded speedup, and that the trailer moves
in such a way that the angles of deviation of the platforms of the trailer relative to
the Chaplygin sleigh have the following
asymptotics:
\begin{gather}\label{ap_phi}
\varphi_i(t) = \frac{\dot{k}}{b} \bigg(c_i +2 \sum\limits_{j=1}^{i-1} c_j\bigg) \Delta^{-2/3} t^{-2/3}  + o(t^{-2/3}), \quad i = 1,\ldots N.
\end{gather}
Consequently, the oscillations of the trailer decay as it moves behind the sleigh.
We also note that the asymptotics
\eqref{eq_pq_tau} do not depend explicitly on the parameters $\mu_i$,  $i=1,\dots N.$

Proposition \ref{pro2} describes the behavior of trajectories with a
fairly large value of velocity $v_1$. Numerical experiments show that, for
trajectories with a small initial value of $v_1$, a speedup arises as well, i.e.,
$v_1\to+\infty$. Figure \ref{pic9} shows dependences of the quantities $t^{-1/3}v_1(t)$,
$\omega(t)$, $\varphi_1(t)$, $\varphi_2(t)$ which have been obtained from the solution of equations \eqref{eq_dvdw} and
\eqref{eq_dphi} for $N = 2$ under the initial conditions
\begin{gather}\label{init_v}
v_1(0) = 10, \quad \omega(0) = 1, \quad \varphi_1(0) = \varphi_2(0) = 0.5,
\end{gather}
for the following parameter values:
\begin{gather}\label{param}
\begin{gathered}
a_0 = 0.7, \quad m_0 = 1,\quad J_0 = 1.5, \quad m_1 = m_2 = 1.2, \\ J_1 = J_2 = 2,  a_1 = 0.1, \quad a_2= 0.2, \\ c_1 = 1.05, \quad c_2= 1.10.		
\end{gathered}
\end{gather}
These dependences are shown as black continuous lines. Also, the red lines in Fig. \ref{pic9}
indicate the estimates of the amplitude above and below on the basis of \eqref{eq.appr} and
\eqref{ap_phi}. In addition, Fig. \ref{pic10} shows the dependence $\varphi_1(t)$
for large times. As can be seen, relations \eqref{eq.appr}
provide a fairly good description of the behavior of the trajectories.

The trajectories of the points of attachment of the wheel pairs to the Chaplygin sleigh
and to the platforms for the three-link vehicle ($N = 2$) with parameters \eqref{param}
and the initial conditions \eqref{init_v}, $\psi(0)=0$, $x(0)=0$, $y(0)=0$ are shown in
Fig. \ref{pic11}. The red line corresponds to the Chaplygin sleigh, and the blue and black
lines correspond to the first and the second platform
of the trailer, respectively.

\begin{figure}[h]
	\centering
	\includegraphics[scale=0.85]{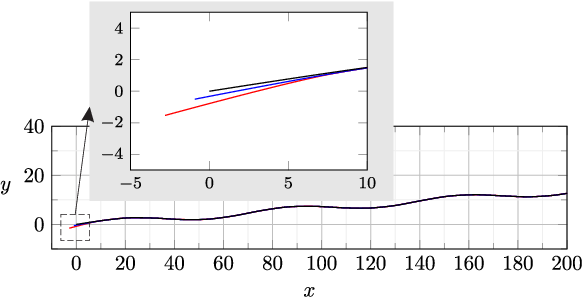}
	\caption{The trajectory of the points of attachment of the wheel pairs in the three-link
		vehicle for the initial conditions \eqref{init_v} and $\psi(0)=0$, $x(0)=0$, $y(0)=0$. The red
		line corresponds to the Chaplygin sleigh, and the blue and black lines correspond to the first
		and the second platform of the
		trailer, respectively. }
	\label{pic11}
\end{figure}

\section{Conclusion}

To conclude, we present the most interesting results and outline avenues for further
research.

In this paper, we have examined the dynamics of an $(N + 1)$-link wheeled vehicle.
For the inertial motion, we have found partial solutions (fixed points, invariant manifolds)
of such a system. For motion in the presence of the rotor of the $(N+1)$-link vehicle we
have shown the existence of trajectories such that one of their velocity components
increases without bound and has the asymptotics $t^{1/3}$. Interestingly,
the deviations of each of the platforms of the trailer relative to the Chaplygin sleigh
(the leading platform) tend to zero.

One avenue of further research is to analyze the motion of a multilink wheeled vehicle
in a more general case, for example, without assumptions {\bf B1} and {\bf C1}.
An open problem is to prove that, apart from the equilibrium point, there are no other regular attractors (limit cycles or tori) in the system \eqref{eqInvariant}.

\section*{Acknowledgements}

The authors extend their gratitude to Alexander A. Kilin, Ivan S. Mamaev and Evgeny
V. Vetchanin for fruitful discussions. The work was supported by the Russian Science
Foundation (No. 21-71-10039).

\bibliographystyle{elsarticle-num}
\bibliography{AB_bib}

\begin{thebibliography}{10}
\expandafter\ifx\csname url\endcsname\relax
  \def\url#1{\texttt{#1}}\fi
\expandafter\ifx\csname urlprefix\endcsname\relax\def\urlprefix{URL }\fi
\expandafter\ifx\csname href\endcsname\relax
  \def\href#1#2{#2} \def\path#1{#1}\fi

\bibitem{borisov2015hadamard}
A.~V. Borisov, A.~A. Kilin, I.~S. Mamaev, On the hadamard--hamel problem and
  the dynamics of wheeled vehicles, Regular and Chaotic Dynamics 20 (2015)
  752--766.

\bibitem{Borisov2015Dynamics}
A.~V. Borisov, A.~A. Kilin, I.~S. Mamaev, Dynamics and control of an omniwheel
  vehicle, Regular and Chaotic Dynamics 20 (2015) 153--172.

\bibitem{Fonseca2020Nonlinear}
L.~M. Fonseca, M.~A. Savi, Nonlinear dynamics of an autonomous robot with
  deformable origami wheels, International journal of non-linear mechanics 125
  (2020) 103533.

\bibitem{chaplygin1912}
S.~A. Chaplygin, On the theory of motion of nonholonomic systems. the
  reducing-multiplier theorem, Regular and Chaotic Dynamics 13 (2008) 369--376.

\bibitem{caratheodory1933}
C.~Carath{\'e}odory, Der schlitten, ZAMM-Journal of Applied Mathematics and
  Mechanics/Zeitschrift f{\"u}r Angewandte Mathematik und Mechanik 13~(2)
  (1933) 71--76.

\bibitem{borisov2009mamaev}
A.~Borisov, I.~Mamaev, The dynamics of a chaplygin sleigh, Journal of Applied
  Mathematics and Mechanics 73~(2) (2009) 156--161.

\bibitem{Czudkova2013Nonholonomic}
L.~Czudkov{\'a}, J.~Musilov{\'a}, Nonholonomic mechanics: A practical
  application of the geometrical theory on fibred manifolds to a planimeter
  motion, International Journal of Non-Linear Mechanics 50 (2013) 19--24.

\bibitem{osborne2005}
J.~M. Osborne, D.~V. Zenkov, Steering the chaplygin sleigh by a moving mass,
  in: Proceedings of the 44th IEEE Conference on Decision and Control, IEEE,
  2005, pp. 1114--1118.

\bibitem{Bizyaev2021Normal}
I.~Bizyaev, S.~Bolotin, I.~Mamaev, Normal forms and averaging in an
  acceleration problem in nonholonomic mechanics, Chaos: An Interdisciplinary
  Journal of Nonlinear Science 31~(1).

\bibitem{sachkov2010maxwell}
I.~Moiseev, Y.~L. Sachkov, Maxwell strata in sub-riemannian problem on the
  group of motions of a plane, ESAIM: Control, Optimisation and Calculus of
  Variations 16~(2) (2010) 380--399.

\bibitem{Ardentov2023OF}
A.~A. Ardentov, Extremals in the markov – dubins problem with control on a
  triangle, Russian Journal of Nonlinear Dynamics.

\bibitem{Karapetyan2019}
A.~Karapetyan, A.~Y. Shamin, On motion of chaplygin sleigh on a horizontal
  plane with dry friction, Mechanics of Solids 54 (2019) 632--637.

\bibitem{Shamin2022}
S.~A. Yu, On the motion of the chaplygin sleigh along a horizontal plane with
  friction in the asymmetric case, Russian Journal of Nonlinear Dynamics 18~(2)
  (2022) 243--251.

\bibitem{fedonyuk2020locomotion}
V.~Fedonyuk, P.~Tallapragada, Locomotion of a compliant mechanism with
  nonholonomic constraints, Journal of Mechanisms and Robotics 12~(5) (2020)
  051006.

\bibitem{Bizyaev2017roller}
I.~A. Bizyaev, The inertial motion of a roller racer, Regular and Chaotic
  Dynamics 22 (2017) 239--247.

\bibitem{BizyaevMamaev2023}
I.~A. Bizyaev, I.~S. Mamaev, Roller racer with varying gyrostatic momentum:
  acceleration criterion and strange attractors, Regular and Chaotic Dynamics
  28~(1) (2023) 107--130.

\bibitem{Krishnaprasad2001}
P.~Krishnaprasad, D.~P. Tsakiris, Oscillations, se (2)-snakes and motion
  control: A study of the roller racer, Dynamical Systems: An International
  Journal 16~(4) (2001) 347--397.

\bibitem{Chakon2017Analysis}
O.~Chakon, Y.~Or, Analysis of underactuated dynamic locomotion systems using
  perturbation expansion: the twistcar toy example, Journal of Nonlinear
  Science 27 (2017) 1215--1234.

\bibitem{Halvani2022Nonholonomic}
O.~Halvani, Y.~Or, Nonholonomic dynamics of the twistcar vehicle: asymptotic
  analysis and hybrid dynamics of frictional skidding, Nonlinear Dynamics
  107~(4) (2022) 3443--3459.

\bibitem{bazzi2017motion}
S.~Bazzi, E.~Shammas, D.~Asmar, M.~T. Mason, Motion analysis of two-link
  nonholonomic swimmers, Nonlinear Dynamics 89 (2017) 2739--2751.

\bibitem{Artemova2020}
E.~M. Artemova, A.~A. Kilin, An integrable case in the dynamics of a three-link
  vehicle, in: 2020 International Conference Nonlinearity, Information and
  Robotics (NIR), IEEE, 2020, pp. 1--6.

\bibitem{Mikishanina2021}
E.~A. Mikishanina, Qualitative analysis of the dynamics of a trailed wheeled
  vehicle with periodic excitation, Russian Journal of Nonlinear Dynamics
  17~(4) (2021) 437--451.

\bibitem{Mikishanina2023problem}
E.~Mikishanina, The problem of acceleration in the dynamics of a double-link
  wheeled vehicle with arbitrarily directed periodic excitation (in press),
  Theoretical and Applied Mechanics~(00) (2023) 9--9.

\bibitem{Artemova2022}
E.~M. Artemova, A.~A. Kilin, Y.~V. Korobeinikova, Investigation of the orbital
  stability of rectilinear motions of roller-racer on a vibrating plane,
  Vestnik Udmurtskogo Universiteta. Matematika. Mekhanika. Komp'yuternye Nauki
  32~(4) (2022) 615--629.

\bibitem{borisov2019invariant}
A.~Borisov, A.~Kilin, I.~Mamaev, Invariant submanifolds of genus 5 and a cantor
  staircase in the nonholonomic model of a snakeboard, International Journal of
  Bifurcation and Chaos 29~(03) (2019) 1930008.

\bibitem{ostrowski1994nonholonomic}
J.~Ostrowski, A.~Lewis, R.~Murray, J.~Burdick, Nonholonomic mechanics and
  locomotion: the snakeboard example, in: Proceedings of the 1994 IEEE
  International Conference on Robotics and Automation, IEEE, 1994, pp.
  2391--2397.

\bibitem{bloch1996nonholonomic}
A.~M. Bloch, P.~Krishnaprasad, J.~E. Marsden, R.~M. Murray, Nonholonomic
  mechanical systems with symmetry, Archive for rational mechanics and analysis
  136 (1996) 21--99.

\bibitem{shammas2007towards}
E.~A. Shammas, H.~Choset, A.~A. Rizzi, Towards a unified approach to motion
  planning for dynamic underactuated mechanical systems with non-holonomic
  constraints, The International Journal of Robotics Research 26~(10) (2007)
  1075--1124.

\bibitem{janova2010streetboard}
J.~Janov{\'a}, J.~Musilov{\'a}, The streetboard rider: an appealing problem in
  non-holonomic mechanics, European journal of physics 31~(2) (2010) 333.

\bibitem{dear2016locomotive}
T.~Dear, S.~D. Kelly, M.~Travers, H.~Choset, Locomotive analysis of a
  single-input three-link snake robot, in: 2016 IEEE 55th Conference on
  Decision and Control (CDC), IEEE, 2016, pp. 7542--7547.

\bibitem{Jing}
F.~Jing, S.~Alben, Optimization of two-and three-link snakelike locomotion,
  Physical Review E 87~(2) (2013) 022711.

\bibitem{Kilin2022}
E.~M. Artemova, A.~A. Kilin, A nonholonomic model and complete controllability
  of a three-link wheeled snake robot, Russian Journal of Nonlinear Dynamics
  18~(4) (2022) 681--707.

\bibitem{shammas2007geometric}
E.~A. Shammas, H.~Choset, A.~A. Rizzi, Geometric motion planning analysis for
  two classes of underactuated mechanical systems, The International Journal of
  Robotics Research 26~(10) (2007) 1043--1073.

\bibitem{Dear}
T.~Dear, S.~D. Kelly, M.~Travers, H.~Choset, The three-link nonholonomic snake
  as a hybrid kinodynamic system, in: 2016 American Control Conference (ACC),
  IEEE, 2016, pp. 7269--7274.

\bibitem{nansai2018generalized}
S.~Nansai, M.~Iwase, H.~Itoh, Generalized singularity analysis of snake-like
  robot, Applied Sciences 8~(10) (2018) 1873.

\bibitem{yona2019wheeled}
T.~Yona, Y.~Or, The wheeled three-link snake model: singularities in
  nonholonomic constraints and stick--slip hybrid dynamics induced by coulomb
  friction, Nonlinear Dynamics 95~(3) (2019) 2307--2324.

\bibitem{Naranjo2015}
A.~Bravo-Doddoli, L.~C. Garc{\'\i}a-Naranjo, The dynamics of an articulated
  n-trailer vehicle, Regular and Chaotic Dynamics 20 (2015) 497--517.

\bibitem{transeth2008modelling}
A.~A. Transeth, Modelling and control of snake robots.

\bibitem{Tanaka}
M.~Tanaka, F.~Matsuno, Modeling and control of head raising snake robots by
  using kinematic redundancy, Journal of Intelligent \& Robotic Systems 75
  (2014) 53--69.

\bibitem{dear2020locomotion}
T.~Dear, B.~Buchanan, R.~Abrajan-Guerrero, S.~D. Kelly, M.~Travers, H.~Choset,
  Locomotion of a multi-link non-holonomic snake robot with passive joints, The
  International Journal of Robotics Research 39~(5) (2020) 598--616.

\bibitem{Kozlov2008Gauss}
V.~V. Kozlov, Gauss principle and realization of constraints, Regular and
  Chaotic Dynamics 13 (2008) 431--434.

\bibitem{Ghori1999Principles}
Q.~Ghori, N.~Ahmed, Principles of lagrange and jacobi for nonholonomic systems,
  International journal of non-linear mechanics 34~(5) (1999) 823--829.

\bibitem{borisov2015symmetries}
A.~V. Borisov, I.~S. Mamaev, Symmetries and reduction in nonholonomic
  mechanics, Regular and Chaotic Dynamics 20 (2015) 553--604.

\bibitem{Hamel}
G.~Hamel, Die lagrange-euler'schen gleichungen der mechanik, Zeitschrift f\"ur
  angewandte Mathematik und Physik 50 (1903) 1--57.

\bibitem{Troger1984Nonlinear}
H.~Troger, K.~Zeman, A nonlinear analysis of the generic types of loss of
  stability of the steady state motion of a tractor-semitrailer, Vehicle System
  Dynamics 13~(4) (1984) 161--172.

\end{thebibliography}

\end{document}